\documentclass[10pt,conference]{ieeeconf}
\IEEEoverridecommandlockouts
\usepackage{cite}
\usepackage{amsmath,amssymb,amsfonts,mathtools}
\usepackage{algorithmic}
\usepackage[dvipsnames]{xcolor}
\usepackage{graphicx}
\usepackage[export]{adjustbox}
\usepackage{tikz}
\usepackage{textcomp}
\usepackage[font=small]{caption}
\def\BibTeX{{\rm B\kern-.05em{\sc i\kern-.025em b}\kern-.08em
		T\kern-.1667em\lower.7ex\hbox{E}\kern-.125emX}}
	
\newcommand{\todo}[1]{\textcolor{red}{To Do: #1}}

\usepackage{amsthm}
\usepackage{mathrsfs}
\usepackage{hyperref}

\makeatletter
\newcommand\footnoteref[1]{\protected@xdef\@thefnmark{\ref{#1}}\@footnotemark}
\makeatother

\newcommand{\reals}{\mathbb{R}}

\DeclareMathOperator*{\argmin}{arg\,min}

\newtheorem{theorem}{Theorem}
\newtheorem{assumption}{Assumption}
\newtheorem{lemma}{Lemma}
\newtheorem{definition}{Definition}
\newtheorem{corollary}{Corollary}

\newtheorem{remark}{Remark}

\usepackage{cancel}
\usepackage[normalem]{ulem} 
\usepackage{verbatim}
\usepackage{url}
\usepackage{hyperref}
\usepackage{placeins}
\usepackage{soul}
\usepackage{tikz}

\definecolor{lcolor}{rgb}{0,0,0.6}
\hypersetup{
	colorlinks=true,
	linkcolor=lcolor,
	filecolor=magenta,      
	urlcolor=black,
}

\newcommand{\regularversion}[1]{\iffalse{}#1\fi}
\newcommand{\extendedversion}[1]{{\color{black}#1}}

\begin{document}

\title{{\fontsize{16}{16} \bf Predictive Control Barrier Functions for Online Safety Critical Control}}%

\author{Joseph Breeden \and Dimitra Panagou\thanks{The authors are with the Department of Aerospace Engineering, University of Michigan, Ann Arbor, MI, USA. Email: \texttt{\{jbreeden,dpanagou\}@umich.edu}.}\thanks{The authors thank the National Science Foundation Graduate Research Fellowship Program for supporting this research.}}

\maketitle

\begin{abstract}
	
This paper presents a methodology for constructing Control Barrier Functions (CBFs) that proactively consider the future safety of a system along a nominal trajectory, and effect corrective action before the trajectory leaves a designated safe set. Specifically, this paper presents a systematic approach for propagating a nominal trajectory on a receding horizon, and then encoding the future safety of this trajectory into a CBF. If the {propagated} trajectory is unsafe, then a controller satisfying the CBF condition will modify the nominal trajectory before the trajectory becomes unsafe. Compared to existing CBF techniques, this strategy is proactive rather than reactive and thus potentially results in smaller modifications to the nominal trajectory. The proposed strategy is shown to be provably safe, and then is demonstrated in simulated scenarios where it would otherwise be difficult to construct a traditional CBF. In simulation, the predictive CBF results in less modification to the nominal trajectory and smaller control inputs than a traditional CBF, and faster computations than a nonlinear model predictive control approach.


\end{abstract}

\vspace{-3pt}
\section{Introduction}

{Control Barrier Functions (CBFs) are a tool for achieving safe control of state-constrained systems in an online fashion \cite{CBF_Tutorial}.}
Specifically, CBFs are functions of the system state that motivate an affine condition on the instantaneous control input that, if satisfied pointwise, guarantees forward invariance of a given set of \textit{safe} states. We call this condition the \textit{CBF condition} and this set the \textit{safe set}. For example, {the state of a car on a shared road should belong to the set of states that are a minimum separation distance from every other car.}

However, one drawback of CBFs is that the CBF condition is reactive (frequently called \emph{myopic} \cite{approximate_optimal_control}); i.e., CBFs only consider the safety of the current state and state derivative, rather than considering future objectives of the system. 
For this reason, the CBF condition may allow trajectories to reach states where large control inputs are required to maintain safety \cite{approximate_optimal_control,multilayer_icra}, or when control inputs are constrained, CBF safe sets are often overly conservative, requiring large evasive actions (e.g., \cite{Automatica}). For instance, CBFs may cause two cars meeting at an intersection to both decelerate to a complete stop to maintain {safety}, whereas it would be more efficient for only one car to decelerate slightly, allowing both cars to {pass} without stopping. To address this and similar situations, we propose the notion of Predictive CBFs (PCBFs). 

Intuitively, PCBFs are functions that encode both the present and future safety of the system in a model-predictive manner, while still providing the convenient CBF condition on the current control input. Specifically, the method in this paper supposes a known nominal trajectory, herein called a \emph{path}, not necessarily safe, that the agent is to follow, and {computes} a measure of the safety of the path on a receding horizon. {This measure is then encoded into a special structure of CBF}. If the {safety} measure is positive {(i.e. unsafe)}, then the resultant CBF condition causes the controller to modify the nominal control input {so as} to render the trajectory safe.

PCBFs are closely related to Model Predictive Control (MPC) \cite{mpc_with_backup_aircraft,multirate_design_cbfs_mpc,Discrete_MPC_CBFs,chemical_control_mpc_cbf,multilayer_icra,discrete_backup_cbf} and to the notion of Backup CBFs \cite{Predictive_CBF,backup_controller,backup_control_multi_robot,CDC21,backup_tutorial,discrete_backup_cbf}, which both also consider safety on a receding horizon, but PCBFs provide distinct advantages. 
Unlike MPC, the CBF condition is always control-affine, even when the dynamics and constraints are nonlinear. The CBF condition also only applies to the current control input rather than to a predicted horizon of control inputs. Thus, the control input under a PCBF can be computed using a Quadratic Program (QP) with small dimension and affine constraints, which can be solved more easily than nonlinear MPC (e.g. \cite{Discrete_MPC_CBFs,multilayer_icra}). Moreover, as the PCBF approach does not introduce a fixed discretization, safety is guaranteed at all times in the prediction horizon rather than just at the MPC sampled times.\extendedversion{{} That said, the PCBF does not yet support sampled-data controllers, as does MPC.}

%

Compared to Backup CBFs, the PCBF works by propagating the system trajectory according to its nominal objectives, which are allowed to be unsafe, rather than according to some backup safety plan. Thus, a backup controller or backup set does not need to be known a priori \cite{backup_controller}, and the system trajectories are not limited to evolving near the backup set, which may be overly conservative. 
\extendedversion{Both Backup CBFs and PCBFs can be constructed from a safety measure of any relative-degree. }%
That said, unlike Backup CBFs, PCBFs do not guarantee input constraint satisfaction, though this is an area for future study. 

Another recent approach for making CBFs consider the future system evolution is to introduce an additional constraint that {bounds the} future control inputs generated by a CBF \cite{feasible_merging_cbfs,feasible_hocbfs}. One can further argue \cite[Sec. III-A]{CDC21} that any Higher Order CBF \cite{Exponential_CBF,hocbf_journal,Automatica} also implies some degree of look-ahead, though this look-ahead is implicit in the choice of class-$\mathcal{K}$ functions and thus may be difficult to tune. The authors in \cite{safe_learning_cbfs_proactive} also found that learned policies tended to proactively avoid unsafe regions better than CBFs alone. This paper supplements all of these works by looking ahead only along a specified path rather than along all possible {safe} trajectories. This work is also closely related to the work in \cite{mitchells_paper}, but is generalized to arbitrary dynamics and paths.

This paper is organized as follows. Section~\ref{sec:preliminaries} introduces preliminaries on CBFs. Section~\ref{sec:method} formalizes the notion of a nominal path (\emph{path function}), introduces a predictive safety function, and proves that this function is a CBF. Section~\ref{sec:simulations} presents simulations and comparisons that show how this function is proactive. Section~\ref{sec:conclusion} presents concluding remarks.

\section{Preliminaries} \label{sec:preliminaries}

\subsection{Notations}

For a differentiable function ${\psi}(t,x)$, let $\partial_t {\psi}(t,x)$ denote the derivative in $t$ 
and let $L_f {\psi}(t,x)$ denote the Lie derivative (directional derivative) along the vector field $f$ at the point $(t,x)$. Denote the total derivative $\frac{d}{dt}{\psi} = \partial_t {\psi} + L_{\dot{x}} {\psi}$. 
%
%
%
Let $\| \cdot \|$ denote the 2-norm. Let $\mathbb{P}'(\mathcal{S})$ denote the set of all subsets of $\mathcal{S}$ {(power set of $\mathcal{S}$)}, and let $\mathbb{P}(\mathcal{S})$ denote the subset of $\mathbb{P}'(\mathcal{S})$ such that each $\mathcal{P}\in\mathbb{P}(\mathcal{S})$ contains finitely many elements. Let $|\mathcal{S}|$ denote the cardinality of set $\mathcal{S}$. A function $\alpha:\reals\rightarrow\reals$ is said to belong to extended class-$\mathcal{K}$, denoted $\alpha \in \mathcal{K}$, if it is strictly increasing and $\alpha(0) = 0$. 
We first note the following consequence of class-$\mathcal{K}$ functions.

\begin{lemma}
\label{lemma:class_k}
Suppose $z : [t_0, \infty) \rightarrow \reals$ is absolutely continuous and satisfies $\dot{z}(t) = \alpha(-z(t))$ for almost every $t\in[t_0,\infty)$, where $\alpha \in \mathcal{K}$. If $z(t_0) \leq 0$, then $z(t) \leq 0$ for all $t \geq t_0$.
\end{lemma}
\begin{proof}
Suppose otherwise; that is{,} suppose $z(t_0) \leq 0$ and that there exists $t > t_0$ such that $z(t) > 0$. This implies that $\dot{z}(t) = \alpha(-z(t)) < 0$. It follows that $z(\tau) > z(t)$ for all $\tau < t$, including $\tau = t_0$. This implies that $z(t_0) > z(t) > 0$, which contradicts the assumption that $z(t_0) \leq 0$. 
\end{proof}


\subsection{Control Barrier Functions}

We consider a system of the form
\begin{equation}
    \dot{x} = f(t,x) + g(t,x) u \label{eq:model}
\end{equation}
with time $t\in\mathcal{T} \triangleq [t_0,t_f]$, state $x \in \mathcal{X} \subseteq \reals^n$, and control input $u \in \reals^m$. Let $f:\mathcal{T}\times\mathcal{X}\rightarrow \reals^n$ and $g:\mathcal{T}\times\mathcal{X}\rightarrow\reals^{n\times m}$ be piecewise continuous in $t$ and locally Lipschitz continuous in $x$. Assume $x(t)$ exists for all $t\in\mathcal{T}$. Given a function $h:\mathcal{T}\times\mathcal{X}\rightarrow\reals$, called a \emph{constraint function}, we {seek to render trajectories always inside} the \emph{safe set} 
\begin{equation}
    \mathcal{S}(t) \triangleq \{ x \in \mathcal{X} \mid h(t,x) \leq 0 \}\,\regularversion{.}\extendedversion{,} \label{eq:safe_set}
\end{equation}
\extendedversion{{specifically by rendering a subset of $\mathcal{S}$} forward invariant.} 
To this end, we first generalize the \regularversion{notion}\extendedversion{definition} of CBF \cite{Automatica,CBF_Tutorial,CBFs_STL} to functions that are differentiable almost everywhere, to account for how our predictive CBFs will have discontinuities at the beginning and end of the prediction horizon.

\begin{definition}
\label{def:cbf}
An absolutely continuous function ${\varphi} : \mathcal{T} \times\mathcal{X} \rightarrow \reals$, denoted ${\varphi}(t,x)$, 
is a \emph{control barrier function (CBF)} 
if there exists $\alpha \in \mathcal{K}$ such that
\begin{equation}
    \inf_{u\in\reals^m} \big[ \underbrace{\partial_t {\hspace{1pt}\varphi}( \cdot ) + L_{f(\cdot) + g(\cdot) u} {\hspace{1pt}\varphi}( \cdot )}_{=\frac{d}{dt}[{\varphi}(t,x)]} \big] \leq \alpha(-{\varphi}(\cdot)) \,, \label{eq:cbf_definition}
\end{equation}
where $(\cdot) = (t,x)$, for almost every $x \in \mathcal{S}_\varphi(t), t \in \mathcal{T}${,
where}
    ${\mathcal{S}_\varphi(t)\triangleq\{x\in\mathcal{X}\mid{\varphi}(t,x)\leq 0 \} \,.}$ 
\end{definition}

\extendedversion{CBFs have gained popularity because they provide for a convenient method to provably guarantee forward invariance of sets of the form \eqref{eq:safe_set}.}

\begin{lemma}
\label{prior:cbf_invariance}
Suppose ${\varphi}:\mathcal{T}\times\mathcal{X}\rightarrow \reals$ is a CBF and let $\alpha\in\mathcal{K}$. 
Then any control law $u(t,x)$ that is piecewise continuous in $t$ and locally Lipschitz continuous in $x$, and that satisfies $u(t,x) \hspace{-0.4pt}\in\hspace{-0.4pt} K_\textnormal{\textrm{cbf}}(t,x)$ for almost every $x\hspace{-0.4pt}\in\hspace{-0.4pt}\mathcal{S}_\varphi(t), t\hspace{-0.4pt}\in\hspace{-0.4pt}\mathcal{T}$, where 
\begin{equation}
    \hspace{-6pt} K_\textnormal{\textrm{cbf}}(\cdot) \hspace{-1.5pt}\triangleq\hspace{-1.5pt} \{ u \hspace{-1.5pt}\in\hspace{-1.5pt} \reals^m \hspace{-1.5pt}\mid\hspace{-1.5pt} \partial_t {\varphi}(\cdot) \hspace{-1.5pt}+\hspace{-1.5pt} L_{f(\cdot) + g(\cdot) u} {\varphi}(\cdot) \hspace{-1.5pt}\leq\hspace{-1.5pt} \alpha(-{\varphi}(\cdot)) \},\hspace{-4pt} \label{eq:cbf_condition}
\end{equation}
will render the set 
{$\mathcal{S}_\varphi$ in Definition~\ref{def:cbf}}
forward invariant.
\end{lemma}
\begin{proof}
The continuity assumptions on $f$, $g$, and $u$ imply that solutions $x(t),t\in\mathcal{T}$ to \eqref{eq:model} are unique. Define ${\theta}(t) = {\varphi}(t,x(t))$. Forward invariance is concerned with states $x(t_0)\in\mathcal{S}_{\varphi}(t_0)$, so ${\theta}(t_0) = {\varphi}(t_0,x(t_0)) \leq 0$. The choice $u(t,x) \in K_\textrm{cbf}(t,x)$ ensures that $\dot{{\theta}}(t) \leq \alpha(-{\theta}(t))$ for almost every $t\in\mathcal{T}$. Thus, by the comparison lemma \cite[Lemma IX.2.6]{nonlipschitz}, ${\theta}(t) \leq \bar{z}(t)$, where $\bar{z}(t)$ is the maximum solution of $\dot{z}(t) = \alpha(-z(t))$ with $z(t_0) = {\theta}(t_0)$. By Lemma~\ref{lemma:class_k}, every solution of $\dot{z}(t) = \alpha(-z(t))$ starting from $z(t_0) = {\theta}(t_0) \leq 0$ satisfies $z(t) \leq 0$ for all $t \in \mathcal{T}$. Therefore, $\bar{z}(t)\leq 0$ for all $t\in\mathcal{T}$, so ${\varphi}(t,x(t)) = {\theta}(t) \leq \bar{z}(t) \leq 0$ and thus $x(t) \in \mathcal{S}_{\varphi}(t)$ for all $t \in \mathcal{T}$, which is the definition of forward invariance of $\mathcal{S}_{\varphi}$.
\end{proof}

\noindent
Note that Lemma~\ref{lemma:class_k} and Lemma~\ref{prior:cbf_invariance} generalize \cite[Lemma 1]{CBFs_STL} and \cite[Thm. 1]{CBFs_STL}, respectively, to allow for non-Lipschitz $\alpha$, which we will need to prove our main theorem in Section~\ref{sec:method}.

We are now ready to introduce the main topic of this paper.

\section{Predictive Control Barrier Functions} \label{sec:method}

The PCBF has three components: a nominal trajectory (Section~\ref{sec:method_a}), points of interest on this trajectory (Section~\ref{sec:method_b}), and the {function structure} (Section~\ref{sec:method_c}). After introducing all three parts, we then analyze the the PCBF zero-sublevel set and present a computational lemma required for implementation. Our main theoretical result then shows that the PCBF satisfies Definition~\ref{def:cbf}. We conclude the section with a discussion of the main theorem assumptions.


\subsection{Path Functions} \label{sec:method_a}

Suppose that an agent seeks to follow a nominal trajectory through $\mathcal{X}$, which is not necessarily safe. We call this {tra-jectory} a \emph{path}, and assume that all potential paths of {a} system can be described by a \emph{path function}, defined as follows.

\begin{definition}
A function $p:\mathcal{T}\times\mathcal{T}\times\mathcal{X}\rightarrow\mathcal{X}$, denoted $p(\tau; t,x)$, is called a \emph{path function} if 1) it is continuously differentiable in all inputs, 2) $p(t; t,x) = x, \forall (t,x)\in\mathcal{T}\times\mathcal{X}$, and 3) for every $(t,x)\in\mathcal{T}\times\mathcal{X}$, there exists a {control law} $\mu:\mathcal{T}\times\mathcal{X}\rightarrow\reals^m$, denoted $\mu(t,x)$, piecewise continuous in $t$ and locally Lipschitz continuous in $x$, that satisfies
\begin{equation}
    \frac{\partial}{\partial \tau}p(\tau; \cdot) = f(\tau, p(\tau; \cdot)) + g(\tau, p(\tau; \cdot)) \mu(\tau, p(\tau; \cdot)) \label{eq:def_path}
\end{equation}
for all $\tau \geq t$, where $(\cdot)=(t,x)$.
\end{definition}

That is, $p$ is a function that from {every $(t,x)$} generates a trajectory $p(\tau; t,x)$ over future times $\tau$ starting from state $x$, where $p$ must be feasible according to the dynamics \eqref{eq:model}. We note that $\tau$ is defined as absolute time in this paper, but $\tau$ could be equivalently defined as time since $t$ (e.g. as in $\beta$ in \cite[Sec. 3.3]{Automatica}). Since generated paths $p(\tau; \cdot)$ are not required to be safe, 
{this paper decouples safety from the computation of $p$, and thus enables the use of very simple nominal paths.} 

\begin{remark}
Path functions can be defined either geometrically (as in Section~\ref{sec:satellite}), provided a feasible $\mu(t,x)$ in \eqref{eq:def_path} exists, or as the explicit (as in Section~\ref{sec:vehicles}) or numerical (as in \cite[Eq. 26]{Automatica}{,\cite[Eq. 23]{backup_controller}}) solution to the differential equation \eqref{eq:def_path} given a nominal input $\mu(t,x)$.
\end{remark}

\subsection{Future Times of Interest on the Propagated Trajectory} \label{sec:method_b}

\begin{figure}
    \raggedleft
    \hspace{-20pt}\begin{tikzpicture}
        \node[anchor=south east,inner sep=0] (image) at (-0.1,2.67) {\includegraphics[width=0.936\columnwidth]{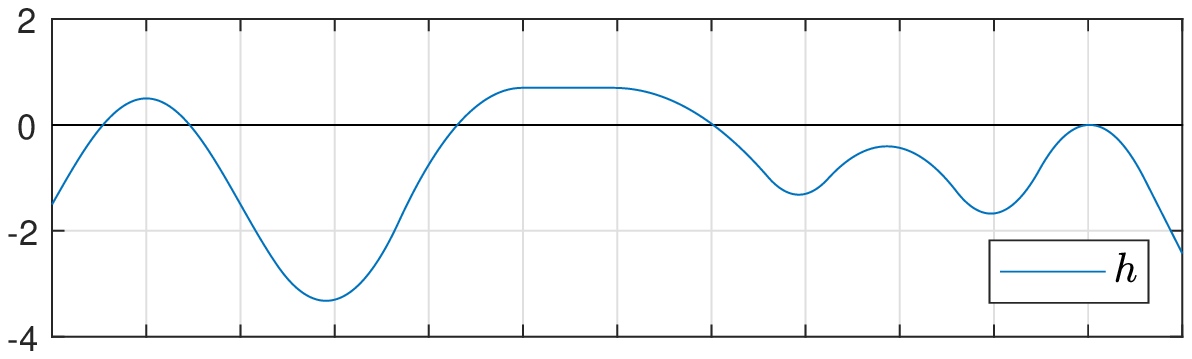}};
        \node[anchor=south east,inner sep=0] (image) at (-0.1,2.1) {\includegraphics[width=0.902\columnwidth]{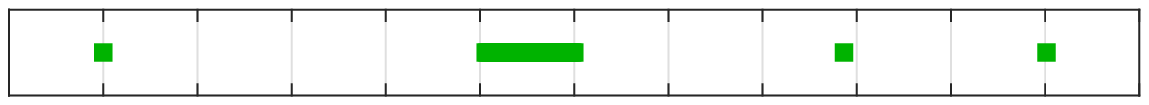}};
        \node[anchor=south east,inner sep=0] (image) at (-0.1,1.44) {\includegraphics[width=0.902\columnwidth]{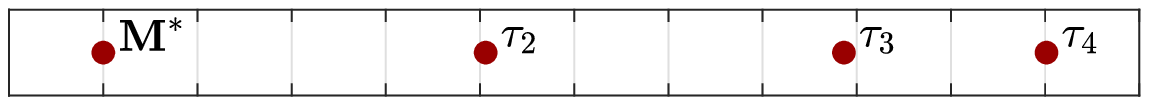}};
        \node[anchor=south east,inner sep=0] (image) at (0,0) {\includegraphics[width=0.92\columnwidth]{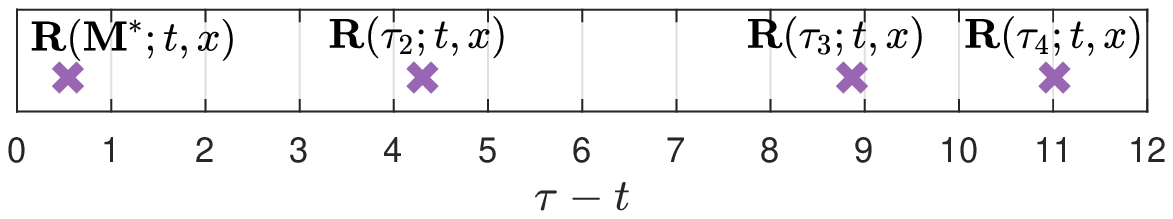}};
        \node [anchor=south east, rotate=90] (note) at (-8.2,5.08) {$h(\tau,p(\tau;t,x))$}; 
        \node [anchor=south east, rotate=90] (note) at (-8.04,2.76) {$\mathscr{M}$}; 
        \node [anchor=south east, rotate=90] (note) at (-8.0,2.12) {$\boldsymbol{M}$}; 
        \node [anchor=south east, rotate=90] (note) at (-7.97,1.5) {$\boldsymbol{R}(\cdot)$}; 
    \end{tikzpicture}
    \caption{A sample trajectory of $h(\tau,p(\tau;t,x))$ with the corresponding sets $\mathscr{M}$ and $\boldsymbol{M}$. The value of $\boldsymbol{R}$ corresponding to each element of $\boldsymbol{M}$ is also shown. The set of all local maximizers $\mathscr{M}$ is dense, so $\boldsymbol{M}$ extracts one element from every isolated subset of $\mathscr{M}$. Each element of $\boldsymbol{M}$ and its corresponding root $\boldsymbol{R}$ is then used to construct \eqref{eq:def_H}-\eqref{eq:def_Hstar}. 
    }
    \label{fig:sets}
    \vspace{-12pt}
\end{figure}

Given a path function, we can examine the safety of the agent forward in time. To do this, we assume {that} the agent {knows both the current and future construction of its safe set}
(e.g. future locations of obstacles) { on a finite receding horizon $T > 0$}, so that {the agent} is able to compute $h$ {across this horizon}. 
%
{Given this, we then reduce the continuous trajectory $p(\tau;\cdot)$ to a finite number of ``times of interest'', namely the maximizers and roots of $h(\tau, p(\tau; \cdot))$. These times are shown for a sample trajectory in Fig.~\ref{fig:sets}, and defined mathematically as follows. First,} define the set of all maximizers within the horizon, ${\mathscr{M}}:\mathcal{T}\times\mathcal{X}\rightarrow \mathbb{P}'(\reals)$, as
\begin{multline}
    \hspace{-10pt} {\mathscr{M}}(t,x) \hspace{-1.5pt}\triangleq\hspace{-1.5pt} \{ \tau \in [t, t+T] \mid \exists \epsilon \hspace{-1.5pt}>\hspace{-1.5pt} 0 : \forall \sigma \hspace{-1.5pt}\in\hspace{-1.5pt} [\tau-\epsilon,\tau+\epsilon] \cap [t, t+T], \\ h(\tau, p(\tau; t,x)) \geq h(\sigma, p(\sigma; t,x)) \} \,. \label{eq:def_M}
\end{multline}
{Note that the set $\mathscr{M}$ could be uncountable, so next define the countable subset} $\boldsymbol{M}:\mathcal{T}\times\mathcal{X}\rightarrow\mathbb{P}(\reals)$ as
\begin{multline}
    \boldsymbol{M}(t,x) \triangleq \{ \tau \in {\mathscr{M}}(t,x) \mid \tau = t \textrm{ or } \\ \exists \epsilon > 0 : \sigma \notin {\mathscr{M}}(t,x), \forall \sigma \in (\tau - \epsilon, \tau) \}, \label{eq:def_Mstar}
\end{multline}
That is, $\boldsymbol{M}$ contains all the isolated elements of ${\mathscr{M}}$ and the first element of every interval in ${\mathscr{M}}${, as shown in Fig.~\ref{fig:sets}}. This extra step is needed to ensure that $\boldsymbol{M}$ has finitely many elements, 
so that $\boldsymbol{M} \in \mathbb{P}(\reals)$. Additionally, let $\boldsymbol{M}^*:\mathcal{T}\times\mathcal{X}\rightarrow\reals$, 
\begin{equation}
     \boldsymbol{M}^*(t,x) \triangleq \min \boldsymbol{M}(t,x) \,, \label{eq:def_M1star}
\end{equation}
denote the first element of $\boldsymbol{M}$, since we are most interested in the soonest future time where the trajectory could leave the safe set.
Next, it is not enough to know just the maximizer times; we also need to know when the predicted paths first become unsafe (if ever), {i.e. the roots of $h(\tau,p(\tau;\cdot))$ over $\tau\in[t,t+T]$,} so define $\boldsymbol{R}:\mathcal{T}\times\mathcal{T}\times\mathcal{X}\rightarrow \reals$ as
\begin{equation}
    \boldsymbol{R}(\tau;t,x) \triangleq \begin{cases} \boldsymbol{R}_1(\tau;t,x) & h(\tau,p(\tau;t,x)) > 0 \\ \tau & h(\tau,p(\tau;t,x)) \leq 0 \end{cases}, \label{eq:def_R}
\end{equation}
\vspace{-10pt}
\begin{multline}
    \hspace{-6pt}\boldsymbol{R}_1(\tau; t,x) \triangleq \max \Big\{ \eta \in [t, \tau] \mid h(\eta, p(\eta; t,x)) = 0, \\ 
    \hspace{-2pt} \exists \epsilon > 0 : h(\sigma, p(\sigma; t,x)) < 0, \forall \sigma \in (\eta-\epsilon, \eta) 
    \Big\}  . \hspace{-8pt} \label{eq:def_R1}
\end{multline}
That is, {for every propagated time $\tau\in[t,t+T]$}, {if $p(\tau;\cdot)$} is unsafe, then $\boldsymbol{R}(\tau; \cdot)$ is the root $\boldsymbol{R}_1=\eta$ of $h(\eta,p(\eta; \cdot))$ immediately preceding $\tau$, or if $p(\tau;\cdot)$ is safe, then 
{let}
$\boldsymbol{R}(\tau; \cdot) = \tau$.
{{The second case of \eqref{eq:def_R} is defined as above because we will be principally interested in $\boldsymbol{R}(\tau;\cdot)$ when $\tau$ is a local maximizer of $h(\tau,p(\tau;\cdot))$. If $\tau \in \boldsymbol{M}(\cdot)$, then it follows that at points where \eqref{eq:def_R} switches cases, i.e. where $h(\tau,p(\tau;\cdot))=0$ (e.g. $\tau_4$ in Fig.~\ref{fig:sets}), both cases of \eqref{eq:def_R} are equivalent. }}%
{Possible trajectories of $h(\tau,p(\tau;\cdot))$ and associated $\boldsymbol{M}$ and $\boldsymbol{R}$ values are shown in Fig.~\ref{fig:explanation}.}
{Together, the local maximums $h(\tau,p(\tau;\cdot)), \tau\in\boldsymbol{M}(\cdot)$ measure how unsafe a predicted trajectory is, and the roots $\boldsymbol{R}(\tau;\cdot)$ measure how much time the controller has to make a correction.}
\begin{figure}
    \centering
    \includegraphics[width=\columnwidth]{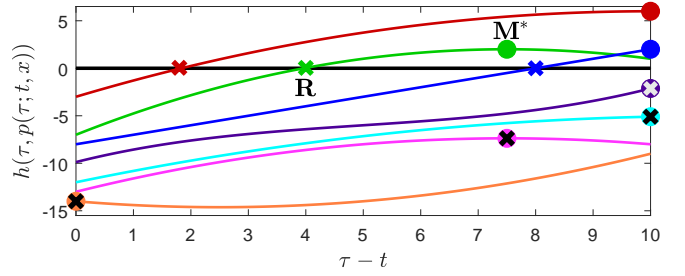}
    \caption{Various possible trajectories of $h(\tau,p(\tau;t,x))$ along a path $p$ on a finite horizon {$T=10$}. Each trajectory shown has only a single maximizer on $[t,t+T]$, so $\boldsymbol{M}^* = \boldsymbol{M}{=\mathscr{M}}$. Each maximizer $\boldsymbol{M}^*$ is annotated with a circle and each root $\boldsymbol{R}$ is annotated with an ``x''. When the maximum value of $h$ is nonpositive (bottom four lines), we let $\boldsymbol{R}(\boldsymbol{M}^*;\cdot) = \boldsymbol{M}^*$.}
    \label{fig:explanation}
    \vspace{-12pt}
\end{figure}

\subsection{Construction of the Predictive Control Barrier Function} \label{sec:method_c}

Next, define {the predictive safety function} $h_p:\mathcal{T}\times\mathcal{T}\times\mathcal{X}\rightarrow\reals$ as
\begin{equation}
    h_p(\tau,t,x) \triangleq h(\tau,p(\tau;t,x)) - m(\boldsymbol{R}(\tau;t,x) - t) \label{eq:def_hp}
\end{equation}
where $m:\reals_{\geq 0}\rightarrow\reals_{\geq 0}$ is any continuously differentiable function that is nondecreasing and satisfies $m(0) = 0$.
We will also need the following assumption.
\begin{assumption}
\label{as:h_max}
Assume that $h(t,x)$ is upper bounded by $h_\textrm{max}<\infty$ for all $t\in\mathcal{T},x\in\mathcal{X}$, and that 
$m(T) \geq h_\textrm{max}$. 
\end{assumption}
The function $h_p$ in \eqref{eq:def_hp} encodes both the future safety {$h$} of the system along $p$, and the time {$\boldsymbol{R}$} at which the path $p$ becomes unsafe. Here, $m$ is a margin to allow for $h_p$ to be nonpositive at $\tau > t$ even if the predicted path is unsafe, i.e. if $h(\tau,p(\tau;t,x)) > 0$. Assumption~\ref{as:h_max} implies that {even if a new maximizer $\tau=t+T$} is discovered at the end of the horizon, $h_p(\tau,t,x)$ will initially be nonpositive. \extendedversion{Our strategy is then to render $h_p(\tau,t,x)$ always nonpositive, which we will show also renders trajectories always inside $\mathcal{S}$ in \eqref{eq:safe_set}.}

\begin{remark}
In our formulation, we choose the argument of $m$ in \eqref{eq:def_hp} to be the time until the {nominal trajectory becomes unsafe \extendedversion{(or the time until $\tau$ if $p(\eta;t,x)$ is safe for all $\eta\in[t,\tau]$)}}. Alternatively, \cite[Eq. 38]{mitchells_paper} chooses the argument of $m$ to be the current safety distance $h(t,x(t))$. We believe the time-based formulation in {\eqref{eq:def_hp}} is more appropriate than \cite[Eq. 38]{mitchells_paper} in cases where $h(\tau,p(\tau;\cdot))$ is expected to vary rapidly with $\tau$ (e.g. as occurs for the system in Section~\ref{sec:satellite}).
\end{remark}

We then apply \eqref{eq:def_hp} at the {local maximizer times} in $\boldsymbol{M}$ in \eqref{eq:def_Mstar} to define functions $H:\mathcal{T}\times\mathcal{X}\rightarrow\mathbb{P}(\reals)$ and $H^*:\mathcal{T}\times\mathcal{X}\rightarrow\reals$ as
\vspace{-5pt}\begin{equation}
    H(t,x) \triangleq \begin{bmatrix} \vdots \\ h_p(\tau_i, t, x)\\ \vdots \end{bmatrix}, \forall \tau_i\in\boldsymbol{M}(t,x), \label{eq:def_H}
\end{equation}\vspace{-5pt}
\begin{equation}
    H^*(t,x) \triangleq h_p(\boldsymbol{M}^*(t,x), t, x) \,. \label{eq:def_Hstar}
\end{equation} \vspace{-14pt} \\
We now present theorems relating the future safety encoded in $H$ and $H^*$ to the present safety given by $h$ in \eqref{eq:safe_set}.


\begin{theorem}
\label{thm:subset}
For some constraint function $h:\mathcal{T}\times\mathcal{X}\rightarrow\reals$, let $\mathcal{S}$ be as in \eqref{eq:safe_set} and $H^*$ as in \eqref{eq:def_Hstar}, with $\mathcal{S}_H^*$ defined as
\begin{equation}
    \mathcal{S}_H^*(t) \triangleq \{ x \in \mathcal{X} \mid H^*(t,x) \leq 0\} \,.
\end{equation}
Then $\mathcal{S}_H^*(t)\subseteq\mathcal{S}(t), \forall t \in \mathcal{T}$.
\end{theorem}
\begin{proof}
By construction, $\boldsymbol{M}$ in \eqref{eq:def_Mstar} is always nonempty, because if no local maxima occur for $\tau \in [t,t+T)$, then $\tau = t+T$ is guaranteed to be an element in \eqref{eq:def_M}. Thus, $\boldsymbol{M}^*$ always exists and $H^*$ is well defined.

First, consider if $\boldsymbol{M}^*(t,x) = t$ (i.e. $h$ is initially decreasing along the path $p$). Then $H^*(t,x) = h(t,x)$ and the result follows immediately. 

Second, consider if $\boldsymbol{M}^*(t,x) \neq t$. This implies that $h$ is initially increasing along the path (e.g., as occurs in the top six lines of Fig.~\ref{fig:explanation}), so it must be that $h(t,x)\leq h(\tau,p(\tau;t,x))$ for all times $\tau \in [t,\boldsymbol{M}^*(t,x)]$, so all we need to do is show that there exists such a time $\tau$ where $h(\tau,p(\tau;t,x)) \leq 0$. If $h(\boldsymbol{M}^*(t,x),p(\boldsymbol{M}^*(t,x);t,x)) \leq 0$ (bottom four lines of Fig.~\ref{fig:explanation}), then the proof is done. If instead $h(\boldsymbol{M}^*(t,x),p(\boldsymbol{M}^*(t,x);t,x)) > 0$ (top three lines of Fig.~\ref{fig:explanation}), then by \eqref{eq:def_hp}, $x\in \mathcal{S}_H^*(t)$ implies that $m(\boldsymbol{R}(\boldsymbol{M}^*(t,x);t,x)-t) > 0$. Since $m(0) = 0$ and $m$  is nondecreasing, it follows that $\boldsymbol{R}(\boldsymbol{M}^*(t,x);t,x) > t$. Thus, there exists $\tau = \boldsymbol{R}(\boldsymbol{M}^*(t,x);t,x) \in (t,\boldsymbol{M}^*(t,x))$ such that $ h(\tau, p(\tau; t,x)) = 0$, so it must be that $h(t,x) \leq h(\tau, p(\tau; t,x)) \leq 0$. Thus, $x \in \mathcal{S}_H^*(t)$ implies $x \in \mathcal{S}(t)$. 
\end{proof}

\begin{corollary}
Let $q(t,x) = | \boldsymbol{M}(t,x) |$ and let ${H}_i$ be the $i$th output of $H$ in \eqref{eq:def_H}. {Define the} set $\mathcal{S}_H$ as
\begin{equation}
    \hspace{-3pt} \mathcal{S}_H(t) \triangleq \{ x \in \mathcal{X} \mid {H}_i(t,x) \leq 0, \forall i = 1, \cdots, q(t,x)\} {\,.} \hspace{-2pt}
\end{equation}
{Then} $\mathcal{S}_H(t) \subseteq \mathcal{S}(t), \forall t \in \mathcal{T}$.
\end{corollary}
\begin{proof}
The result follows immediately since  
${H}_1 \equiv H^*$. 
\end{proof}

That is, rendering either $\mathcal{S}_H$ or $\mathcal{S}_H^*$ forward invariant is sufficient to render trajectories always inside the safe set $\mathcal{S}$. The remaining question is under what conditions rendering $\mathcal{S}_H$ or $\mathcal{S}_H^*$ forward invariant will be possible. First, we make some regularity assumptions.

\begin{assumption}
\label{as:differentiable}
Assume that the functions $H$ and $H^*$ in \eqref{eq:def_H} and \eqref{eq:def_Hstar}, respectively, are absolutely continuous. Assume further that each maximizer location $\tau_i \in \boldsymbol{M}(t,x) \cap (t,t+T)$, where $\tau_i = \tau_i(t,x)$, is continuously differentiable.
\end{assumption}

That is, each maximizer $\tau_i$ in $\boldsymbol{M}(t,x)$ is also a function of $t$ and $x$, and we assume that the variation of $\tau_i(t,x)$ is regular when $\tau_i \in (t,t+T)$. When $\tau_i$ is at the start or end of the horizon, this will in general not be true, which is why we only assume $H$ and $H^*$ to be absolutely continuous rather than continuously differentiable. 

Next, we present a result on how to compute the derivatives of $H^*$ in \eqref{eq:def_Hstar} and each output ${H}_i$ in \eqref{eq:def_H}.

\begin{lemma}
\label{lemma:computing}
Let $\tau \in \boldsymbol{M}(t,x)$ 
and let $m':\reals_{\geq0}\rightarrow\reals_{\geq0}$ be the derivative of $m$ in \eqref{eq:def_hp}. Assume $x \in \mathcal{S}_H^*(t)$. 

\emph{Case i:} If $\tau \in (t,t+T)$, then
\vspace{-2pt}
\begin{multline}
\frac{d}{dt}[h_p(\tau, p(\tau; t, x))] = \\ m'(\boldsymbol{R}(\tau;t,x)-t) + B(\tau,t,x) (u - \mu(t,x)) \,. \label{eq:computational_theorem}
\end{multline}

\emph{Case ii:}
If $\tau = t+T$ and $\boldsymbol{R}(\tau;t,x) < \tau$, then 
\vspace{-2pt}
\begin{multline}
    \frac{d}{dt}[h_p(\tau;p(\tau;t,x))] = \frac{d}{d\tau}[h(\tau,p(\tau;t,x))] \frac{d \tau(t,x)}{dt} \\ + m'(\boldsymbol{R}(\tau;t,x)-t) + B(\tau,t,x)(u - \mu(t,x)) \,. \label{eq:computational_theorem2}
\end{multline}

\emph{Case iii:}
If $\tau \in \{t, t+T\}$ and $\boldsymbol{R}(\tau;t,x) = \tau$, then
\begin{align}
    \frac{d}{dt}[h_p(\tau;&p(\tau;t,x))] = \frac{d}{d\tau}[h(\tau,p(\tau;t,x))] \frac{d \tau(t,x)}{dt} \nonumber  \\ & + \frac{\partial h(\tau,p(\tau;t,x))}{\partial x}\frac{\partial p(\tau;t,x)}{\partial x} g(t,x)(u - \mu(t,x)) \nonumber \\ & -m'(\boldsymbol{R}(\tau;t,x) - t) \left( \frac{d \tau(t,x)}{dt} - 1\right) \,, \label{eq:computational_theorem3} \hspace{-4pt}
\end{align}
\vspace{-12pt}

\noindent
where
\begin{multline}
    B(\tau;t,x) = \bigg[ \frac{\partial h(\tau,p(\tau;t,x))}{\partial x}\frac{\partial p(\tau;t,x)}{\partial x} \\ - m'(\boldsymbol{R}(\tau;t,x) - t) C(\tau,t,x) \bigg] g(t,x) \,, \label{eq:def_B}
\end{multline}
\begin{equation}
    \hspace{-3pt} C(\tau,t,x) = \begin{cases} \frac{\partial \tau(t,x)}{\partial x} & \boldsymbol{R}(\tau;t,x) = \tau \\ C_1(\boldsymbol{R}_1(\tau;t,x),t,x) & \boldsymbol{R}(\tau;t,x) \neq \tau \end{cases} , \hspace{-2pt} \label{eq:def_C}
\end{equation}
\vspace{-6pt}
\begin{multline}
    \hspace{-10pt} C_1(\eta,t,x) = - \frac{\partial h(\eta,p(\eta;t,x))}{\partial x} \frac{\partial p(\eta;t,x)}{\partial x}\\ \cdot \left[ \frac{\partial h(\eta,p(\eta;t,x))}{\partial t}  + \frac{\partial h(\eta,p(\eta;t,x))}{\partial x}\frac{\partial p(\eta;t,x)}{\partial \tau} \right]^{-1} \hspace{-8pt}\hspace{-6pt} \label{eq:def_C1}
\end{multline}
and where $\frac{d\tau(t,x)}{dt} = \partial_t \tau(t,x) + L_{f(t,x) + g(t,x) u} \tau(t,x)$ describes the sensitivity of the maximizer time to the current time $t\in\mathcal{T}$ and state $x\in\mathcal{X}$.
\end{lemma}
\begin{proof}
See appendix.
\end{proof}

We are now ready to prove that, under mild assumptions elaborated upon after the proof, $H^*$ in \eqref{eq:def_Hstar} is a CBF, and thus can be used to ensure safety via the CBF condition.

\begin{theorem}
\label{thm:is_a_cbf}
Let the derivative $m':\reals_{\geq0}\rightarrow\reals_{\geq0}$ of $m$ in \eqref{eq:def_hp} be strictly positive on $(0,T)$. Suppose that there exists $\gamma$ such that $\frac{d}{d\tau}[h(\tau,p(\tau;t,x))] \leq \gamma$ for all $\tau\in\mathcal{T},t\in\mathcal{T},x\in\mathcal{X}$. Suppose also that $\| \frac{\partial h(\eta, p(\eta; t,x))}{\partial x}\frac{\partial p(\eta; t,x)}{\partial x} g(t,x) \| \neq 0$ for all $\eta \in (t,t+T)\setminus {\mathscr{M}}(t,x)$ and for all $t\in\mathcal{T},x\in\mathcal{X}$, and that $\frac{\partial h(\tau, p(\tau;t,x))}{\partial x}^\textnormal{\textrm{T}} \frac{\partial h(\eta, p(\eta;t,x))}{\partial x} \geq 0$ whenever $\eta = \boldsymbol{R}(\tau;t,x)$ for all $\tau\in\mathcal{T},t\in\mathcal{T},x\in\mathcal{X}$. Then $H^*$ in \eqref{eq:def_Hstar} is a CBF.
\end{theorem}
\begin{proof}
By Assumption~\ref{as:differentiable}, $H^*$ is absolutely continuous, and therefore differentiable almost everywhere. Thus, $H^*$ is a CBF {as in Definition~\ref{def:cbf}} if there exists $\alpha\in\mathcal{K}$ satisfying condition \eqref{eq:cbf_definition} almost everywhere. We will show that such an $\alpha\in\mathcal{K}$ always exists, by breaking this proof into cases where A) $\boldsymbol{R}(\boldsymbol{M}^*(t,x);t,x) = \boldsymbol{M}^*(t,x)$ and B) $\boldsymbol{R}(\boldsymbol{M}^*(t,x);t,x) \neq \boldsymbol{M}^*(t,x)$. We further break both A and B into the cases 1) $\boldsymbol{M}^*(t,x) = t$, 2)  $\boldsymbol{M}^*(t,x) \in (t, t+T)$, and 3) $\boldsymbol{M}^*(t,x) = t+T$. Note that while Lemma~\ref{lemma:computing} applies to any $\tau \in \boldsymbol{M}(t,x)$, in this theorem, we only consider $\tau = \boldsymbol{M}^*(t,x)$, so $\frac{d\tau(t,x)}{dt} \equiv \frac{d\boldsymbol{M}^*(t,x)}{dt}$ in this proof.

In cases A1-A3, the path $p(\tau;t,x)$ is safe for {future times} $\tau$ {until at least $\boldsymbol{M}^*(t,x)$}, so intuitively choosing $u = \mu(t,x)$ should render the system safe, thereby satisfying \eqref{eq:cbf_definition}. We show this formally as follows.
First, case A1 (orange line in Fig.~\ref{fig:explanation}) implies that $\tau = \boldsymbol{M}^*(t,x) = t$ is a local maximizer of $h(\tau,p(\tau;t,x))$ on $\tau \in [t,t+T]$. This implies that $h$ is initially nonincreasing along the path $p$ beginning from $(t,x)$, so $\frac{d}{d\tau}\left.[h(\tau,p(\tau;t,x)]\right|_{\tau=t} \leq 0$. Moreover, if the agent continues along the path (i.e. chooses $u = \mu(t,x)$), then at an arbitrarily small time $t+\epsilon$ in the future, $\boldsymbol{M}^*(t + \epsilon,p(t+\epsilon;t,x)) = t+\epsilon$ will be the first local maximizer on the horizon $[t+\epsilon,t+\epsilon+T]$ {(see also $\mathcal{B}_{loc}(0)$ in \cite[Fig.~2]{Automatica})}, so $\frac{d\boldsymbol{M}^*(t,x)}{dt} = 1$ under $u = \mu(t,x)$. Thus, \eqref{eq:computational_theorem3} tells us 
that $\frac{d}{dt}[H^*(t,x)] =\frac{d}{d\tau}\left.[h(\tau,p(\tau;t,x)]\right|_{\tau=t}\leq 0$ under $u = \mu(t,x)$, which satisfies condition \eqref{eq:cbf_definition} for any $\alpha \in \mathcal{K}$.

Next, in case A2 (magenta line in Fig.~\ref{fig:explanation}), \eqref{eq:computational_theorem} implies that under $u = \mu(t,x)$, it follows that $\frac{d}{dt}[H^*(t,x)] = m'(\boldsymbol{M}^*(t,x) - t)$. Moreover, case A assumes that $\boldsymbol{R}(\boldsymbol{M}^*(t,x);t,x) = \boldsymbol{M}^*(t,x)$, which by \eqref{eq:def_R} implies that the predicted maximum value $h(\boldsymbol{M}^*(t,x),p(\boldsymbol{M}^*(t,x);t,x))$ is at most zero, so $H^*(t,x) \leq -m(\boldsymbol{M}^*(t,x) - t) < 0$.
Thus, choose $\alpha \in \mathcal{K}$ such that $\alpha(m(\lambda)) \geq m'(\lambda), \forall \lambda \in [0,T]$, and then \eqref{eq:cbf_definition} is satisfied under $u = \mu(t,x)$. Note that if $h(\boldsymbol{M}^*(t,x),p(\boldsymbol{M}^*(t,x);t,x)) = 0$, then $u = \mu(t,x)$ will cause $H^*$ to reach (but not exceed since $\boldsymbol{M}^*$ is a maximizer) the origin in finite time. Thus, a function $\alpha \in \mathcal{K}$ satisfying the above condition will not be Lipschitz continuous (recall that a Lipschitz $\alpha$ will not allow finite-time convergence to the origin), which is why we needed Lemma~\ref{lemma:class_k}.

Next, in case A3, two behaviors are possible at the end of the horizon. First, if $h$ is nonincreasing along the path at the end of the horizon (cyan line in Fig.~\ref{fig:explanation}), i.e. $\frac{d}{d\tau}\left. [h(\tau,p(\tau;t,x))]\right|_{\tau = t+T} = 0$, then under $u = \mu(t,x)$, at some arbitrarily small time $t+\epsilon$ in the future, $\boldsymbol{M}^*(t+\epsilon,p(t+\epsilon;t,x)) = \boldsymbol{M}^*(t,x)$ will still be a local maximizer. Thus, $\frac{d\boldsymbol{M}^*(t,x)}{dt} = 0$ under $u = \mu(t,x)$, so \eqref{eq:computational_theorem3} implies that $\frac{d}{dt}[H^*(t,x)] = m'(\boldsymbol{M}^*(t,x) - t)$, and the same logic as in case A2 implies satisfaction of condition \eqref{eq:cbf_definition}. Second, if instead $\frac{d}{d\tau}\left. [h(\tau,p(\tau;t,x))]\right|_{\tau = t+T} > 0$ (purple line in Fig.~\ref{fig:explanation}), then under $u = \mu(t,x)$, at some arbitrarily small time $t+\epsilon$ in the future, $\boldsymbol{M}^*(t+\epsilon,p(t+\epsilon;t,x)) = \boldsymbol{M}^*(t,x)+\epsilon$ will be the local maximizer on the horizon $[t+\epsilon,t+\epsilon+T]$. Thus,  $\frac{d\boldsymbol{M}^*(t,x)}{dt} = 1$ under $u = \mu(t,x)$, so \eqref{eq:computational_theorem3} implies that $\frac{d}{dt}[H^*(t,x)] = \frac{d}{d\tau}\left.[h(\tau,p(\tau;t,x))]\right|_{\tau=t+T} \leq \gamma$. Similar to case A2, $H^*(t,x) \leq -m(\boldsymbol{M}^*(t,x) - t) = -m(T) < 0$, so choose $\alpha$ so that $\alpha(m(T)) \geq \gamma$, and then condition \eqref{eq:cbf_definition} is satisfied for $u = \mu(T,x)$.

Next, note that case B1 is not possible, since \eqref{eq:def_R} implies that $\boldsymbol{R}(t;t,x) = t$, which would imply case A1. 
Next, cases B2-B3 
imply that $h(\boldsymbol{M}^*(t,x),p(\boldsymbol{M}^*(t,x);t,x)) > 0$; i.e the path $p(\tau;t,x)$ becomes unsafe on the horizon $[t,\boldsymbol{M}^*]\subseteq [t,t+T]$, so unlike cases A1-A3, we will ensure safety by taking a control action $u \neq \mu(t,x)$. Since $h(\boldsymbol{M}^*(t,x),p(\boldsymbol{M}^*(t,x);t,x)) > 0$, any $x \in \mathcal{S}_H^*(t)$ must satisfy $\boldsymbol{R}(\boldsymbol{M}^*(t,x);t,x) > t$, so we have until time $\boldsymbol{R}(\boldsymbol{M}^*(t,x);t,x)$ to take corrective action.
First, in case B2, both the maximizer time $\boldsymbol{M}^*$ and the root $\boldsymbol{R}$ occur within the time horizon (green line in Fig.~\ref{fig:explanation}), and $\frac{d}{dt}[H^*(t,x)]$ is given by \eqref{eq:computational_theorem}. Since $m'$ is assumed to be strictly positive and $\frac{\partial h(\eta, p(\eta; t,x))}{\partial x}\frac{\partial p(\eta; t,x)}{\partial x} g(t,x)$ is assumed to not be the zero vector, the second line of \eqref{eq:def_B} is guaranteed nonzero. Additionally, since we assumed the inner product of $\frac{\partial h(\tau, p(\tau;t,x))}{\partial x}$ and $\frac{\partial h(\eta, p(\eta;t,x))}{\partial x}$ is nonnegative for $\tau = \boldsymbol{M}^*(t,x)$ and $\eta = \boldsymbol{R}(\boldsymbol{M}^*(t,x);t,x)$, the first line of \eqref{eq:def_B} cannot cancel with the second line, so the matrix $B(\tau;t,x)$ in \eqref{eq:def_B} is guaranteed nonzero. Since $u \in \reals^m$ is unconstrained, we can thus choose $u$ to make $\frac{d}{dt}[H^*(t,x)]$  in \eqref{eq:computational_theorem} arbitrarily negative, so \eqref{eq:cbf_definition} can be satisfied for any $\alpha \in \mathcal{K}$.

Lastly, case B3 represents the scenario where the path becomes unsafe for some time $\boldsymbol{R}(\boldsymbol{M}^*(t,x);t,x) \in (t,t+T)$ in the prediction horizon, and the maximizer of $h(\tau,p(\tau;t,x))$ occurs either at (red line in Fig.~\ref{fig:explanation}) or after (blue line in Fig.~\ref{fig:explanation}) the end of the horizon. For this case to occur, it must be that $\frac{d}{d\tau}\left. [h(\tau,p(\tau;t,x))] \right|_{\tau=t+T} \geq 0$. However, by assumption, $\frac{d}{d\tau}[h(\tau,p(\tau;t,x))] \leq \gamma$ also. Since our prediction horizon $T$ is constant, $\frac{d\boldsymbol{M}^*(t,x)}{dt} \leq 1$, so the first line of \eqref{eq:computational_theorem2} is upper bounded by $\gamma$. The term $m'(\boldsymbol{R}(\tau;t,x) - t)$ is also independent of $u$ and is bounded since $m$ is assumed to be continuously differentiable. Therefore, by the same argument as in case B2, we can choose $u$ to make $\frac{d}{dt}[H^*(t,x)]$ in \eqref{eq:computational_theorem2} arbitrarily negative, so \eqref{eq:cbf_definition} can be satisfied for any $\alpha \in \mathcal{K}$.

Thus, in every case, a function $\alpha\in\mathcal{K}$ satisfying $\alpha(m(\lambda)) \geq m'(\lambda), \forall \lambda \in [0,T]$ and $\alpha(m(T)) \geq \gamma$ will satisfy condition \eqref{eq:cbf_definition}, so $H^*$ is indeed a CBF. Note that these are sufficient, not necessary conditions on $\alpha$.
\end{proof}

That is, under mild assumptions, it is always possible to render the set $\mathcal{S}_H^*$ forward invariant, so we call $H^*$ as in \eqref{eq:def_Hstar} a PCBF. To elaborate on these assumptions, {first, $\gamma$ represents a bound on the rate of change of $h$ along the nominal path $p$, which is a reasonable assumption for most paths. Next,} intuitively, the assumption that $\| \frac{\partial h(\eta, p(\eta; t,x))}{\partial x}\frac{\partial p(\eta; t,x)}{\partial x} g(t,x) \| \neq 0$ is analogous to controllability and observability. 
This quantity is the sensitivity of the future value of $h$ to the current control input, so this assumption encodes that both the state trajectory $p(\eta;\cdot)$ and the observed values $h(\eta,p(\eta;\cdot))$ along that trajectory must be controllable.
If instead {future} $h$ {values are} not controllable along $p$, then we cannot expect such a predictive strategy using the path function $p$ to be useful. That said, we make an exception and allow for $\| \frac{\partial h(\tau, p(\tau; t,x))}{\partial x}\frac{\partial p(\tau; t,x)}{\partial x} g(t,x) \| = 0$ to occur when $\tau$ is a maximizer time {(i.e. $\tau\in\mathscr{M}(t,x)$)}. This is because one can easily construct a path for which the maximum value $\max_{\tau \in [t,t+T]}h(\tau,p(\tau;t,x))$ of $h$ is constant with $t$ and $x$; for example, a path that attempts to converge to an unsafe state. In this case, no control input will change the maximum value of $h$ along the path, but the assumption that $\| \frac{\partial h(\eta, p(\eta; t,x))}{\partial x}\frac{\partial p(\eta; t,x)}{\partial x} g(t,x) \| \neq 0$ for all $\eta \notin \boldsymbol{M}(t,x)$ encodes that it is still possible to delay the time $\boldsymbol{R}(\tau;t,x)$ at which the state trajectory first becomes unsafe by choosing a control input $u$ different from the nominal input $\mu(t,x)$.

The related assumption that $\frac{\partial h(\tau, p(\tau;t,x))}{\partial x}^\textnormal{\textrm{T}} \frac{\partial h(\eta, p(\eta;t,x))}{\partial x} \geq 0$ means that if the path is such that it is possible to change both 1) the maximum value of $h$ achieved along the path (first term of \eqref{eq:def_hp}) and 2) the first time along the path at which the trajectory becomes unsafe (second term of \eqref{eq:def_hp}), then there exists a control action which decreases both terms of \eqref{eq:def_hp} simultaneously. For example, a car moving towards a stop sign and decelerating at $|u|$ larger than $|\mu|$ will stop both \textit{shorter} and \textit{sooner} than the same car decelerating at $|\mu|$. That is, this assumption excludes the use of dynamics that would cause the car to instead stop \textit{further} and \textit{sooner}, or \textit{shorter} and \textit{later}.

In particular, we note that we have not made any assumptions on the relative-degree of $h$. For many systems, such as those in Section~\ref{sec:simulations}, this method will generate a CBF $H^*$ even when $h$ is not a CBF. \extendedversion{That is, using only the constraint function $h$, the dynamics $f$ and $g$, and a nominal, not necessarily safe, control law $\mu(t,x)$, one can directly utilize the sensitivities \eqref{eq:computational_theorem}-\eqref{eq:computational_theorem3} of the nominal trajectory $p$ to construct a PCBF and compute a safe trajectory.} Note also that Theorems~\ref{thm:subset}-\ref{thm:is_a_cbf} do not assume that the system actually follows the ``nominal'' trajectory $p(\tau;t,x)${, even when this trajectory is safe}. Rather, this trajectory is just used to generate the PCBF. However, we only expect a controller to proactively ensure safety better than traditional CBFs when $H^*$ is constructed using a path that is relevant to the system's {intended} behavior. Thus, one possible safe controller is
\begin{equation}
    u(t,x) = \argmin_{u \in K_\textrm{cbf}(t,x)} \| u -\mu(t,x) \|^2 \,, \label{eq:the_qp}
\end{equation}
where {$K_\textrm{cbf}$ in \eqref{eq:cbf_condition} uses $\varphi = H^*$ and where} we assume that \eqref{eq:the_qp} is sufficiently regular for Lemma~\ref{prior:cbf_invariance}.

Next, while $H^*$ in \eqref{eq:def_Hstar} is sufficient for enforcing safety, to make the controller even more proactive, one may want to consider future safety for all elements of $H$ in \eqref{eq:def_H} simultaneously. However, it is in general difficult to show that $H$ in \eqref{eq:def_H} is also a CBF (with appropriate generalizations for $H$ being set-valued), because there may not exist control inputs such that the CBF condition can be satisfied simultaneously for every output ${H}_i$. That said, it may still be advantageous to construct the CBF condition for every future time-of-interest and use slack penalties for violation of the CBF condition for $\tau_i\in\boldsymbol{M}\setminus\boldsymbol{M}^*$ in case of a conflict. \extendedversion{Additionally, instead of enforcing the CBF condition on every ${H}_i$, one can reduce the size of $H$ by instead considering the maximal value of $H$ \cite{Glotfelter2}, the smooth maximum of $H$ \cite{CBFs_STL}, or by using hysteresis tolerances to make certain elements of $H$ inactive \cite{Automatica}.}

\extendedversion{\begin{remark}
To avoid needing to run a line-search for maximizers and roots in \eqref{eq:def_M}-\eqref{eq:def_R1}, one could instead sample $\boldsymbol{M}$ as many discrete time points in the finite horizon (e.g. every $\Delta T$ time point), similar to MPC. With a suitable robustness margin for the discretization, this could also be used to guarantee safety in a similar manner. This strategy would increase the number of constraints on the control input, but would still be simpler than MPC approaches because only the current control input is {an optimization variable in \eqref{eq:the_qp}}.
\end{remark}}

We now present simulations demonstrating the utility of the PCBF $H^*$.

\section{Simulations} \label{sec:simulations}

\subsection{Autonomous Vehicles} \label{sec:vehicles}

Our first case study involves two cars passing through a four-way intersection. We assume that the cars are fixed in their lanes $l_1:\reals\rightarrow\reals^2$ and $l_2:\reals\rightarrow\reals^2$, respectively, with locations $z_1$ and $z_2$ along their lanes. Thus, the position of car~1 on the road is $l_1(z_1)$ and the position of car~2 is $l_2(z_2)$. For simplicity, we model the cars as double-integrators: $\ddot{z}_1 = u_1$ and $\ddot{z}_2 = u_2$, resulting in state vector $x = [z_1;\, \dot{z}_1;\, z_2;\, \dot{z}_2]$ and control vector $u = [u_1; u_2]$. Suppose the cars nominally want to travel in their lanes at velocities $v_1$ and $v_2$, respectively, so the nominal control input is $\mu = [\mu_1; \mu_2]$ where $\mu_i(t,x) = k(v_i - \dot{z}_i)$ for some gain $k > 0$. The path function is then $p = [p_1;\,p_2]$ where $p_i$ is the solution to \eqref{eq:def_path} under $\mu$:
\begin{equation}
    \hspace{-4pt} p_i(\tau;t,x) = \begin{bmatrix} z_i + v_i(\tau - t) + \frac{\dot{z}_i - v_i}{k}(1 - e^{-k(\tau - t)}) \\ v_i + (\dot{z}_i - v_i)e^{-k(\tau-t)} \end{bmatrix}. \hspace{-4pt}
\end{equation}
Let the safety constraint be $h = \rho - \|l_1(z_1) - l_2(z_2)\|$.

\begin{figure}
    \centering
    \includegraphics[width=\columnwidth]{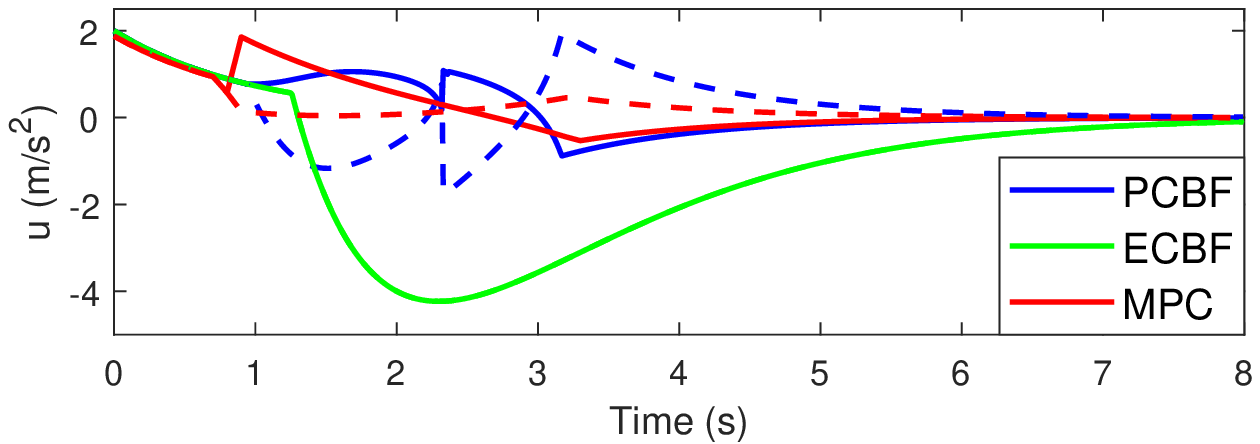}
    \caption{Control inputs of two vehicles with safety determined by an ECBF, PCBF, or by MPC. The solid lines are $u_1$ (acceleration of vehicle 1) and the dashed lines are $u_2$ (acceleration of vehicle 2).}
    \label{fig:control}
\regularversion{\vspace{10pt}}\extendedversion{\vspace{7pt}}
    \centering
    \includegraphics[width=\columnwidth]{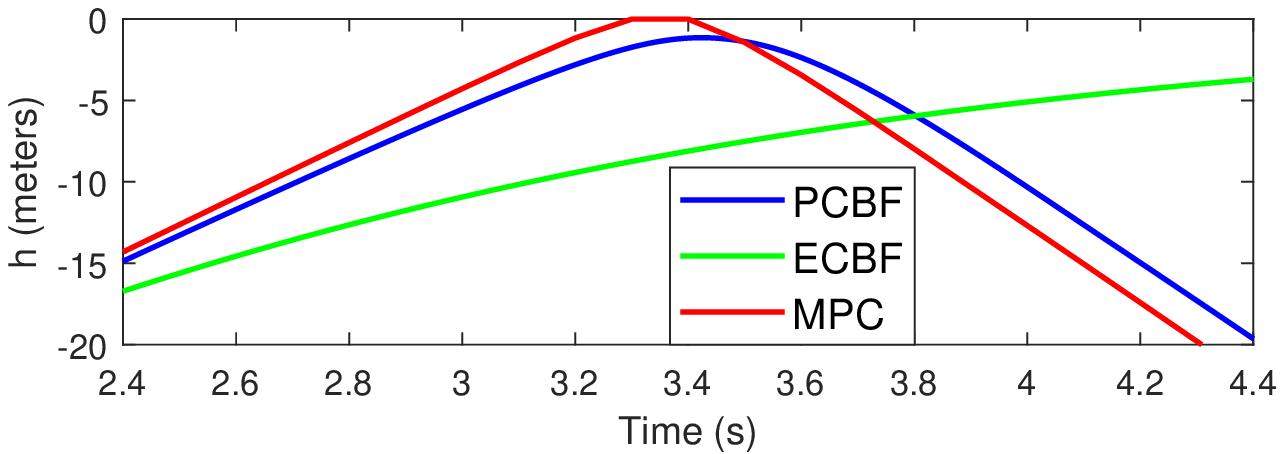}
    \caption{The values of $h$ during the three simulations in Fig.~\ref{fig:control}. The PCBF and MPC trajectories are similar, whereas the ECBF trajectory slowly converges to zero as the vehicles come to a complete stop. {See also the animation\textsuperscript{\protect\hyperlink{note:car}{1}}.}}
    \label{fig:safety}
    \vspace{-12pt}
\end{figure}

We then constructed a PCBF of the form $H^*$, and simulated two intersection scenarios using the controller \eqref{eq:the_qp}, one with two cars intersecting perpendicularly, and one with one car driving straight and a second car turning left {and passing through $l_1$}. The results were very similar, so we focus on the left-turn case. For comparison, we also simulated this case with safety governed by an Exponential CBF \cite{Exponential_CBF} and via nonlinear MPC using the same prediction horizon. The control inputs are shown in Fig.~\ref{fig:control} and the safety constraint is shown in Fig.~\ref{fig:safety}. Videos of every simulation\footnote{\hypertarget{note:car}{\url{https://youtu.be/0tVUAX6MCno}}} and all parameters and simulation code\footnote{\label{note:code}\url{https://github.com/jbreeden-um/phd-code/tree/main/2022/CDC\%20Predictive\%20CBFs}} are available below. From Figs.~\ref{fig:control}-\ref{fig:safety}, we see that the PCBF and MPC performed generally similarly, with both cars approaching close and then continuing on opposite sides of the intersection. On the other hand, the ECBF caused both agents to come to a complete stop, so neither agent made it through the intersection. On average, the computation time per control computation was 0.0034~s for the ECBF, 0.050~s for the PCBF, and 0.98~s for the MPC approach, {all on a 2.6 GHz laptop CPU, though the code\footnoteref{note:code} for all three cases could likely be further optimized}.

\subsection{Satellite Collision Avoidance} \label{sec:satellite}

Our second case study involves two satellites in Low Earth Orbit avoiding collision.  Let $r_1$ and $r_2$ denote the satellite positions and let $h(t,x) = \rho - \| r_1 - r_2 \|$ where $\rho = 1\textrm{ km}$. Let the satellites be governed by two body dynamics. This scenario is not well solved by existing CBF methods, or gradient methods in general, because the satellites necessarily orbit at large velocities, so $h$ may rapidly change from very negative to zero. If the satellites wait to take evasive maneuvers until $h$ is nearly zero, then large control actions will be necessary to maintain safety. However, very small control actions taken approximately a half orbit preceding a predicted collision (when $h$ is most negative) will result in large changes to $r_1$ and/or $r_2$ at the predicted collision time. Therefore, we expect the strategy in this paper to perform better than traditional CBFs.\extendedversion{\footnotetext[3]{{\hypertarget{note:satellite}{\url{https://youtu.be/HhtWUG63BWY}}}}}

\begin{figure}
    \centering
    \includegraphics[width=\columnwidth]{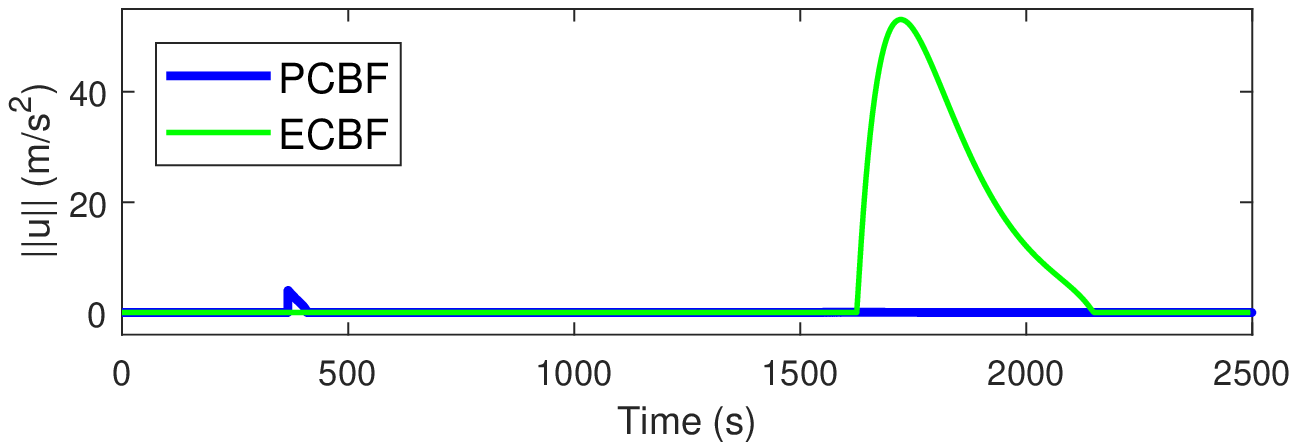}
    \caption{Control thrusts of a satellite with safety determined by an ECBF or PCBF.}
    \label{fig:control_space}
%
\regularversion{\vspace{10pt}}\extendedversion{\vspace{7pt}}
    \centering
    \begin{tikzpicture}
        \node[anchor=south west,inner sep=0] (image) at (0,0) {\includegraphics[width=\columnwidth]{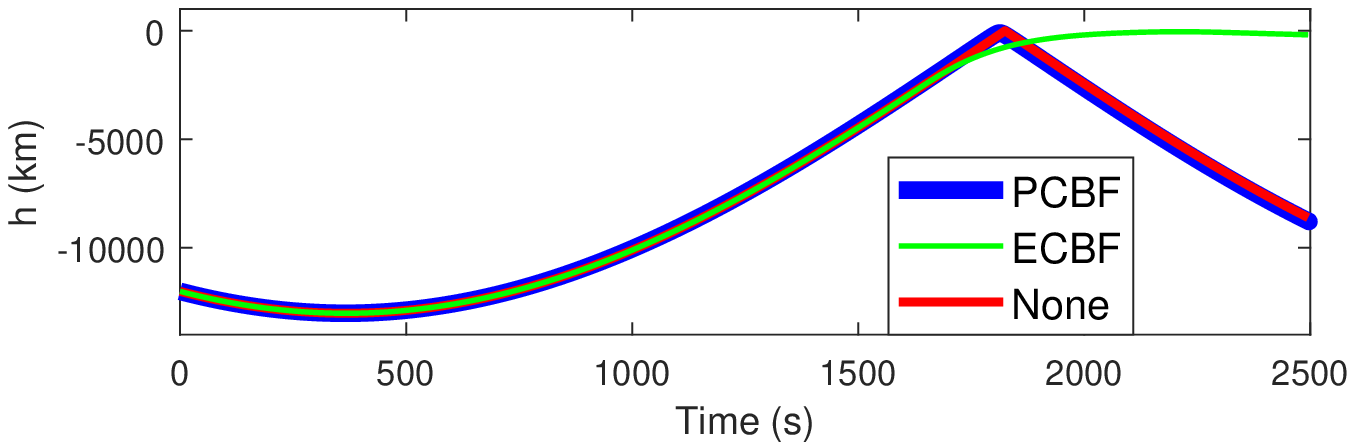}};
        \node[anchor=south west,inner sep=0] (image) at (1.25,1.05) 
        {\includegraphics[width=0.75in,trim={0in, .46in, 0in, .14in},frame,clip]{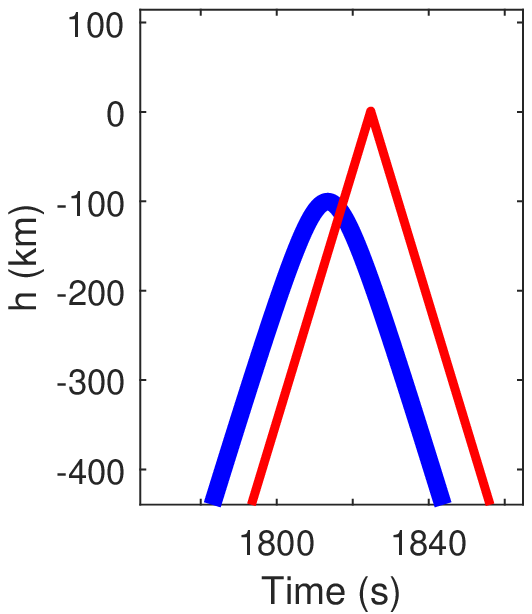}}; 
    \end{tikzpicture}
    \caption{The values of $h$ during the two simulations in Fig.~\ref{fig:control_space} and a simulation with zero control input. The PCBF minimally modifies the original trajectory so the red and blue lines are very similar, but the red line exceeds the safe set by 1 km (see zoomed in inset {or the animation\textsuperscript{\protect\hyperlink{note:satellite}{3}}}), while {the} blue and green lines remain safe.}
    \label{fig:safety_space}
    \vspace{-12pt}
\end{figure}

For this case study, let $x = [r_1; \dot{r}_1]$ represent the state of a controlled satellite and let $r_2$ be an uncontrolled piece of debris with a known orbit. Let $p$ be the trajectory of the controlled satellite under zero control input{, which is a well-characterized elliptical orbit}. {The simulation scenario places the controlled satellite and the debris initially very far apart, but in orbits that eventually intersect if no control action is taken.} Simulations under $H^*$, under an ECBF, and under no control action are shown in Figs.~\ref{fig:control_space}-\ref{fig:safety_space} {and in the video below\regularversion{\footnote{\hypertarget{note:satellite}{\url{https://youtu.be/HhtWUG63BWY}}}}}\extendedversion{\hyperlink{note:satellite}{\footnotemark[3]}}. Note how the PCBF trajectory (blue) takes a small control action as soon as the unsafe prediction enters the horizon at $t=367\textrm{ s}$, and then is very similar to the nominal trajectory (red). On the other hand, the ECBF trajectory (green) takes control action much later, when over 10 times as much thrust is required. This avoidance problem could {in theory also} be solved with MPC, but would require a very fine discretization, because $h > 0$ occurs for only 0.14 seconds since the satellites are moving so fast.  {{Thus,} utilizing the same length of prediction horizon} {would require more than $10^4$ samples, making the MPC problem (and most other online trajectory search algorithms) intractable.}


\section{Conclusion} \label{sec:conclusion}

We have developed a systematic approach for propagating state trajectories into the future along a nominal path, and then adjusting the trajectory if the propagated path is found to be unsafe. The approach reduces the degree to which the CBF condition intrudes on the nominal path compared to a traditional CBF in simulation, and performs similarly to an MPC approach while reducing computation times by an order of magnitude and avoiding the need for linearizations. {The satellite collision avoidance simulation in particular shows how the PCBF can also solve problems that are not well handled by either traditional CBFs or MPC.} Future work includes extending this approach to consider input constraints, and extending the safety guarantees to collections of agents with distributed controllers.

\bibliographystyle{ieeetran}
\bibliography{sources}

\section{Appendix} 
\regularversion{\allowdisplaybreaks}
\begin{proof}[Proof of Lemma~\ref{lemma:computing}]
First, note that every $\tau$ in $\boldsymbol{M}(t,x)$ is a function $\tau(t,x)$ defined as the maximizer of $h$ along the path originating from $(t,x)$ and restricted to a sufficiently small domain around $\tau$ (see \cite[Fig. 2]{Automatica} for a related visualization). 

We start by recognizing that the computation of $h(\tau,p(\tau;t,x))$ can be broken into two parts by an intermediate time $\sigma$ as follows
\begin{equation}
    h(\tau, p(\tau; t,x)) = h(\tau, p(\tau; \sigma, p(\sigma; t,x))) . \label{eq:h_sigma}
\end{equation}
Differentiate both sides of \eqref{eq:h_sigma} w.r.t $\sigma$ and set $\sigma = t$ to get
\begin{equation}
    0 = \frac{\partial h(\tau, p(\tau; \cdot))}{\partial x}\left[ \frac{\partial p(\tau; \cdot)}{\partial t} + \frac{\partial p(\tau; \cdot)}{\partial x} \frac{\partial p(t;\cdot)}{\partial \tau} \right] \,,  \label{eq:maximizer_nominal}
\end{equation}
where $(\cdot) = (t,x)$. 
We then compute the total derivative of the first term of \eqref{eq:def_hp} as\regularversion{{\cite{extended_version}}}
\extendedversion{follows
\begin{align}
    & \hspace{-8pt}\frac{d h(\tau, p(\tau; t,x))}{dt} = \frac{\partial h(\tau,p(\tau;\cdot))}{\partial t} \frac{d \tau}{d t} \nonumber \\ &\;\;\;\;\; + \frac{\partial h(\tau,p(\tau;\cdot))}{\partial x}\left[ \frac{\partial p(\tau;\cdot)}{\partial \tau}\frac{d \tau}{dt} + \frac{\partial p(\tau;\cdot)}{\partial t} + \frac{\partial p(\tau;\cdot)}{\partial x}\frac{dx}{dt}\right] \nonumber \\
 	&\hspace{-2pt}\overset{\eqref{eq:maximizer_nominal}}{=} \left[ \frac{\partial h(\tau,p(\tau;\cdot))}{\partial t} + \frac{\partial h(\tau,p(\tau;\cdot))}{\partial x}\frac{\partial p(\tau;\cdot)}{\partial \tau} \right] \frac{d \tau}{dt} \nonumber \\ &\;\;\; + \frac{\partial h(\tau,p(\tau;t,x))}{\partial x}\frac{\partial p(\tau;t,x)}{\partial x}\left[ \frac{dx}{dt} - \frac{\partial p(t;t,x)}{\partial \tau} \right] \nonumber \\
 	&\hspace{-2pt}\overset{\eqref{eq:model},\eqref{eq:def_path}}{=} \frac{d}{d\tau}[h(\tau,p(\tau;t,x))] \frac{d \tau}{dt}\nonumber \\ &\;\;\; + \frac{\partial h(\tau,p(\tau;t,x))}{\partial x}\frac{\partial p(\tau;t,x)}{\partial x} g(t,x) (u - \mu(t,x)). \hspace{-3pt} \label{eq:final_form_local}
\end{align}}%
\regularversion{\begin{align}
    & \hspace{-8pt}\frac{d h(\tau, p(\tau; t,x))}{dt} =  \frac{d}{d\tau}[h(\tau,p(\tau;t,x))] \frac{d \tau}{dt}\nonumber \\ &\;\;\; + \frac{\partial h(\tau,p(\tau;t,x))}{\partial x}\frac{\partial p(\tau;t,x)}{\partial x} g(t,x) (u - \mu(t,x)). \hspace{-3pt} \label{eq:final_form_local}
\end{align}
{For the full steps of the simplifications in this proof, see \cite[Sec.~VI]{extended_version}.}}
Note that \eqref{eq:final_form_local} holds in all three cases in Lemma~\ref{lemma:computing}.

Next, consider the second term of \eqref{eq:def_hp}, which has total derivative
\extendedversion{\begin{multline}
    \frac{d m(\boldsymbol{R}(\tau;t,x) - t)}{dt} = \\ m'(\boldsymbol{R}(\tau;t,x) - t) \left( \frac{d \boldsymbol{R}(\tau;t,x)}{dt} - 1\right) \,, \label{eq:diff_second_term}
\end{multline}}%
\regularversion{\begin{equation}
    \frac{d m(\boldsymbol{R}(\tau;t,x) - t)}{dt} \hspace{-1pt}=\hspace{-1pt} m'(\boldsymbol{R}(\tau;t,x) - t) \left( \frac{d \boldsymbol{R}(\tau;t,x)}{dt} - 1\right), \label{eq:diff_second_term}
\end{equation}}%
for which we need to know $\frac{d}{dt}[\boldsymbol{R}(\tau;t,x)]$. {Note that $x\in\mathcal{S}_H^*(t)$ implies that $\boldsymbol{R}(\tau;t,x)$ is well defined for all $\tau \geq t$.}
In the case where $\boldsymbol{R}(\tau;t,x) = \tau$, then $\frac{d}{dt}[\boldsymbol{R}(\tau;t,x)] = \frac{d\tau(t,x)}{dt}$. 
If instead $\boldsymbol{R}(\tau;t,x) \neq \tau$, then by \eqref{eq:def_R}, $\eta = \boldsymbol{R}_1(\tau;t,x)$ is a zero of $h(\eta,p(\eta;t,x))$ over $\eta$ in a neighborhood preceding $\tau$. That is,
\begin{equation}
    0 = h(\boldsymbol{R}_1(\tau;t,x),p(\boldsymbol{R}_1(\tau;t,x);t,x)) \,,
\end{equation}
which we can differentiate with respect to $t$ to get
\begin{multline}
    \hspace{-10pt}\frac{\partial h(\boldsymbol{R}_1,p(\boldsymbol{R}_1;t,x))}{\partial t} \frac{d \boldsymbol{R}_1}{dt} + \frac{\partial h(\boldsymbol{R}_1,p(\boldsymbol{R}_1;t,x))}{\partial x} \hspace{-1.5pt}\left[ \frac{\partial p(\boldsymbol{R}_1;t,x)}{\partial \tau} \right. \\ \left. \cdot\, \frac{d\boldsymbol{R}_1}{dt}  + \frac{\partial p(\boldsymbol{R}_1;t,x)}{\partial t} + \frac{\partial p(\boldsymbol{R}_1;t,x)}{\partial x} \frac{dx}{dt} \right] = 0\label{eq:diff_R} 
\end{multline}
and then simplify\regularversion{\cite{extended_version}} to find
\extendedversion{\begin{align}
    &\hspace{-10pt} \frac{d \boldsymbol{R}_{1\hspace{-1pt}}(\tau;t,x)}{d t} \hspace{-1.5pt}\overset{\eqref{eq:maximizer_nominal},\eqref{eq:diff_R}}{=}\hspace{-1.5pt} -  \frac{\partial h(\boldsymbol{R}_{1\hspace{-1pt}},p(\boldsymbol{R}_{1\hspace{-1pt}};\cdot))}{\partial x} \frac{\partial p(\boldsymbol{R}_{1\hspace{-1pt}};\cdot)}{\partial x} \hspace{-1.5pt}\left[ \frac{dx}{dt} \hspace{-1.5pt}-\hspace{-1.5pt} \frac{\partial p(t;\cdot)}{\partial \tau} \right] \nonumber \\ & \;\;\;\;\; \cdot \left( \frac{\partial h(\boldsymbol{R}_{1\hspace{-1pt}},p(\boldsymbol{R}_{1\hspace{-1pt}};\cdot))}{\partial t}  + \frac{\partial h(\boldsymbol{R}_{1\hspace{-1pt}},p(\boldsymbol{R}_{1\hspace{-1pt}};\cdot))}{\partial x}\frac{\partial p(\boldsymbol{R}_{1\hspace{-1pt}};\cdot)}{\partial \tau} \right)^{-1} \nonumber \\ 
    &\hspace{-10pt}\overset{\eqref{eq:model},\eqref{eq:def_path}}{=} - \frac{\partial h(\boldsymbol{R}_{1\hspace{-1pt}},p(\boldsymbol{R}_{1\hspace{-1pt}};\cdot))}{\partial x} \frac{\partial p(\boldsymbol{R}_{1\hspace{-1pt}};\cdot)}{\partial x} g(t,x)(u - \mu(t,x)) \nonumber \\ &\; \cdot \left( \frac{\partial h(\boldsymbol{R}_{1\hspace{-1pt}},p(\boldsymbol{R}_{1\hspace{-1pt}};\cdot))}{\partial t}  \hspace{-1.5pt}+\hspace{-1.5pt} \frac{\partial h(\boldsymbol{R}_{1\hspace{-1pt}},p(\boldsymbol{R}_{1\hspace{-1pt}};\cdot))}{\partial x}\frac{\partial p(\boldsymbol{R}_{1\hspace{-1pt}};\cdot)}{\partial \tau} \right)^{-1} \hspace{-8pt} \label{eq:final_dR}
\end{align}}%
\regularversion{\begin{align}
    &\hspace{-10pt} \frac{d \boldsymbol{R}_{1\hspace{-1pt}}(\tau;t,x)}{d t} \hspace{-1.5pt}= - \frac{\partial h(\boldsymbol{R}_{1\hspace{-1pt}},p(\boldsymbol{R}_{1\hspace{-1pt}};\cdot))}{\partial x} \frac{\partial p(\boldsymbol{R}_{1\hspace{-1pt}};\cdot)}{\partial x} g(t,x)(u - \mu(t,x)) \nonumber \\ &\; \cdot \left( \frac{\partial h(\boldsymbol{R}_{1\hspace{-1pt}},p(\boldsymbol{R}_{1\hspace{-1pt}};\cdot))}{\partial t}  \hspace{-1.5pt}+\hspace{-1.5pt} \frac{\partial h(\boldsymbol{R}_{1\hspace{-1pt}},p(\boldsymbol{R}_{1\hspace{-1pt}};\cdot))}{\partial x}\frac{\partial p(\boldsymbol{R}_{1\hspace{-1pt}};\cdot)}{\partial \tau} \right)^{-1} \hspace{-8pt} \label{eq:final_dR}
\end{align}}%
where \eqref{eq:final_dR} is used to define $C_1(\boldsymbol{R}_1,t,x)$ in \eqref{eq:def_C1}. 

The total derivative $\frac{d}{dt}[h(\tau(t,x),p(\tau(t,x);t,x))]$ in \eqref{eq:computational_theorem}-\eqref{eq:computational_theorem3} is then the difference between \eqref{eq:final_form_local} and \eqref{eq:diff_second_term} where we make the following simplifications.

\emph{Case i:} In this case, $\tau$ is a maximizer of $h(\tau,p(\tau;t,x))$ on an open interval, and $h$ and $p$ are assumed to be continuously differentiable, so a necessary condition for $\tau \in \boldsymbol{M}(t,x)$ is that $\frac{d}{d\tau}[h(\tau,p(\tau;t,x))] = 0$. Therefore, the first line of \eqref{eq:final_form_local} is zero in Case i and does not appear in \eqref{eq:computational_theorem}. The remaining second line of \eqref{eq:final_form_local} is incorporated into the first line of \eqref{eq:def_B}.

Next, since $\tau \in (t,t+T)$, by Assumption~\ref{as:differentiable}, $\tau(t,x)$ is continuously differentiable. Similar to \eqref{eq:h_sigma}, we can introduce the intermediate time $\sigma$ where
\begin{equation}
    \tau(t,x) = \tau(\sigma,p(\sigma;t,x)) \,. \label{eq:tau_sigma}
\end{equation}
Differentiate both sides of \eqref{eq:tau_sigma} w.r.t $\sigma$ and set $\sigma = t$ \regularversion{similar to \eqref{eq:h_sigma}-\eqref{eq:final_form_local}, and then simplify \cite{extended_version}} to get
\extendedversion{\begin{equation}
    0 = \frac{\partial \tau(t,x)}{\partial t} + \frac{\partial \tau(t,x)}{\partial x}\frac{\partial p(t;t,x)}{\partial \tau} \,. \label{eq:tau_nominal}
\end{equation}
We can then describe the variation in the maximizer location $\tau$ with respect to the current time and state as
\begin{align}
    \frac{d\tau(t,x)}{dt} &= \frac{\partial \tau(t,x)}{\partial t} + \frac{\partial \tau(t,x)}{\partial x}\frac{dx}{dt} \nonumber \\
    &\overset{\eqref{eq:tau_nominal}}{=} \frac{\partial \tau(t,x)}{\partial x}\left[ \frac{dx}{dt} - \frac{\partial p(t;t,x)}{\partial \tau} \right] \nonumber \\ 
    &\overset{\eqref{eq:model},\eqref{eq:def_path}}{=} \frac{\partial \tau(t,x)}{\partial x}g(t,x)(u - \mu(t,x)) \,. \label{eq:diff_tau}
\end{align}}%
\regularversion{\begin{equation}
    \frac{d\tau(t,x)}{dt} = \frac{\partial \tau(t,x)}{\partial x}g(t,x)(u - \mu(t,x)) \,. \label{eq:diff_tau}
\end{equation}}%
Note that \eqref{eq:diff_tau} is only valid if $\tau(t,x)$ is continuously differentiable, which is why we need different formulas for Cases ii-iii.
When $\boldsymbol{R}(\tau;t,x) = \tau$, the simplification \eqref{eq:diff_tau} yields the first case of $C(\tau,t,x)$ in \eqref{eq:def_C}. When instead $\boldsymbol{R}(\tau;t,x) \neq \tau$, the simplification \eqref{eq:final_dR} yields the second case of $C(\tau,t,x)$ in \eqref{eq:def_C}. Both cases are then incorporated into the second line of \eqref{eq:def_B}. Thus, \eqref{eq:computational_theorem} accounts for all the terms of the difference between \eqref{eq:final_form_local} and \eqref{eq:diff_second_term}.

\emph{Case ii:} In this case, $\frac{d}{d\tau}[h(\tau,p(\tau;t,x))]$ is no longer guaranteed to be zero, so we add the first line of \eqref{eq:final_form_local} back to \eqref{eq:computational_theorem2} compared to \eqref{eq:computational_theorem}. Since $\tau$ is an endpoint of the horizon, $\tau$ is no longer assumed continuously differentiable so $\frac{\partial \tau(t,x)}{\partial x}$ may no longer be well-defined. Thus, the simplification in \eqref{eq:diff_tau} may no longer hold, so we do not simplify $\frac{d\tau(t,x)}{dt}$ from \eqref{eq:final_form_local} any further. Next, note that since Case~ii assumes $\boldsymbol{R}(\tau;t,x) < \tau$, the second case of \eqref{eq:def_C} (first case of \eqref{eq:def_R}) will always hold in Case~ii (recall {that $\frac{\partial \tau(t,x)}{\partial x}$ in} the first case of \eqref{eq:def_C} may no longer be defined), so we can use the same definition of $B(\tau,t,x)$ as in Case~i. 
Thus, \eqref{eq:computational_theorem2} accounts for all the terms of the difference between \eqref{eq:final_form_local} and \eqref{eq:diff_second_term}.

\emph{Case iii:} This case is identical to Case ii, except that $\boldsymbol{R}(\tau;t,x) = \tau$ so $\frac{d}{dt}[\boldsymbol{R}(\tau;t,x)]$ is equivalent to $\frac{d\tau(t,x)}{dt}$ instead of coming from \eqref{eq:final_dR} via $C(\tau,t,x)$ in \eqref{eq:def_C}. Thus, the first two lines of \eqref{eq:computational_theorem3} accounts for all the terms of \eqref{eq:final_form_local}, while the last line of \eqref{eq:computational_theorem3} accounts for all the terms of \eqref{eq:diff_second_term}.
\end{proof}

\end{document}